\documentclass[twoside,12pt]{article}
\usepackage[T1]{fontenc}
\usepackage[utf8]{inputenc}
\usepackage{mathrsfs}
\usepackage{amssymb,amsmath,mathrsfs,amsthm}
\usepackage{mathtools}
\usepackage{graphicx}
\usepackage{color}
\usepackage[top=2cm,bottom=2cm,left=2cm,right=2cm]{geometry}
\usepackage{float,caption}
\usepackage{microtype}

\input{epsf}

\newcommand{\bd}{\begin{description}}
\newcommand{\ed}{\end{description}}
\newcommand{\bi}{\begin{itemize}}
\newcommand{\ei}{\end{itemize}}
\newcommand{\be}{\begin{enumerate}}
\newcommand{\ee}{\end{enumerate}}
\newcommand{\beq}{\begin{equation}}
\newcommand{\eeq}{\end{equation}}
\newcommand{\beqs}{\begin{eqnarray*}}
\newcommand{\eeqs}{\end{eqnarray*}}

\definecolor{DarkGreen}{rgb}{0.2,0.6,0.3}

\newtheorem{theorem}{Theorem}[section]

\newtheorem{lemma}[theorem]{Lemma}
\newtheorem{definition}[theorem]{Definition}
\newtheorem{corollary}[theorem]{Corollary}

\allowdisplaybreaks

\newcommand{\oR}{\overline{\mathrm{R}}}
\newcommand{\oM}{\overline{\mathrm{M}}}
\newcommand{\oRM}{\overline{\mathrm{RM}}}
\newcommand{\cK}{\mathcal{K}}
\newcommand{\cS}{\mathcal{S}}
\newcommand{\cG}{\mathcal{G}}
\newcommand{\cH}{\mathcal{H}}
\newcommand{\bB}{\mathfrak{B}}
\newcommand{\Imin}{I_*}
\newcommand{\hmin}{h_*}
\newcommand{\Jt}{J_t}
\newcommand{\blambda}{\boldsymbol{\lambda}}
\newcommand{\lambdamin}{\lambda_*}

\begin{document}

\title{\textbf{Ramsey multiplicity for ordered graphs\footnote{Supported by the National Natural Science Foundation of China (Nos.~12471329 and 12061059).}}}

\author{Mengya He\footnote{School of Mathematics and Statistics, Qinghai Normal University, Xining, Qinghai 810008, China. {\tt hmy8536@163.com}},\ \ Yaping Mao\footnote{Academy of Plateau Science and Sustainability and School of Mathematics and Statistics, Qinghai Normal University, Xining, Qinghai 810008, China. {\tt yapingmao@outlook.com}},\ \ Bing Wei\footnote{Department of Mathematics, University of Mississippi, University,  MS 38677, USA. {\tt bwei@olemiss.edu}},\ \ Qinghong Zhao\footnote{Corresponding author: School of Mathematical Sciences, Huaqiao University, Quanzhou 362021, China. {\tt qzhao@hqu.edu.cn}}}
\date{}
\maketitle
\begin{abstract}
Let \(\cG_1,\ldots,\cG_k\) be fixed vertex-ordered graphs, each containing
at least one edge. The ordered Ramsey number
\(\oR(\cG_1,\ldots,\cG_k)\) is the least integer \(N\) such that every
\(k\)-edge-coloring of the ordered complete graph \(\cK_N\) contains an
order-preserving copy of \(\cG_i\) in color \(i\) for some \(i\in[k]\).
For positive weights
\(\blambda=(\lambda_1,\ldots,\lambda_k)\), let
\(\oM_{\blambda}(n;\cG_1,\ldots,\cG_k)\) denote the minimum weighted number of correctly
colored, order-preserving copies of the target graphs over all
\(k\)-edge-colorings of \(\cK_n\). When \(\blambda=\bf{1}\), \(\oM_{\bf{1}}(n;\cG_1,\ldots,\cG_k)=\oM(n;\cG_1,\ldots,\cG_k)\) is called the ordered Ramsey multiplicity. 
In this paper, we first establish the amplification inequality
\[
\oM_{\blambda}(n;\cG_1,\ldots,\cG_k)
\ge
\oM_{\blambda}(t;\cG_1,\ldots,\cG_k)
\frac{\binom{n}{\hmin}}{\binom{t}{\hmin}},
\]
where $h_i=v(\cG_i),\hmin=\min_{i\in[k]}h_i$, and $n\ge t\ge\oR(\cG_1,\ldots,\cG_k)$. 
Let $\cS_{r,s}$ be the ordered star whose center has $r-1$ leaves to its left and $s-1$ leaves to its right, and
let $\bB_m$ be the family of all ordered perfect matchings on $[2m]$ containing the edge $\{1,2m\}$. We apply the
amplification inequality to obtain the multiplicity lower bounds for ordered stars and ordered perfect matchings. We then obtain the upper bound
$\oM_{\boldsymbol\lambda}
(n;\cS_{r_1,s_1},\cS_{r_2,s_2})
\le
\min\{\lambda_1 B_{h_1}(n),\lambda_2 B_{h_2}(n)\}$ by constructions,
where $B_{h_i}(n):=
\binom{\lfloor n/2\rfloor}{h_i}
+
\binom{\lceil n/2\rceil}{h_i}$ and $h_i=r_i+s_i-1$ for $i\in [2]$. We also derive a random-coloring upper bound for ordered stars and prove
\[\oM(n; \bB_m,\bB_m)
\le
\binom{n}{2m}
\frac{(2m-2)!}{2^{2m-2}(m-1)!}.\]
Finally, we establish a regularity-based lifting theorem for ordered colorings.
\\[2mm]
\end{abstract}

\section{Introduction}

Ramsey theory studies when a prescribed monochromatic structure is
unavoidable in every edge-coloring of a sufficiently large complete graph. For a positive integer $k$, write $[k]=\{1,\ldots,k\}$. Given graphs \(G_1,\ldots,G_k\), the Ramsey number
\(R(G_1,\ldots,G_k)\) is the smallest integer \(N\) such that every
\(k\)-edge-coloring of \(K_N\) contains a copy of \(G_i\) in color \(i\) for
some \(i\in[k]\). The classical Ramsey theorem~\cite{Ramsey} and the
monograph~\cite{GrahamRothschildSpencer} provide general background on this
topic.

Once the existence of a monochromatic copy is guaranteed, it is natural to
ask how many such copies must occur. For a graph \(G\) without isolated
vertices, the multiplicity \(M(G;n)\) is defined as the minimum number of
monochromatic copies of \(G\) over all red-blue edge-colorings of \(K_n\).
The corresponding Ramsey multiplicity is obtained by considering the host
graph with order equal to the Ramsey number. The first result on
multiplicity was obtained by Goodman~\cite{Goodman}, who determined the
minimum number of monochromatic triangles. Harary and Prins~\cite{HararyPrins}
introduced the notion of Ramsey multiplicity and studied it systematically.
The early results and open problems were surveyed by Burr and
Rosta~\cite{BurrRosta}. The threshold multiplicity for paths and cycles~\cite{ConlonFoxSudakovWeiOdd,
ConlonFoxSudakovWeiEven, Huang}, off-diagonal multiplicity
\cite{HydeLeeNoel,MossNoel}, and computational approaches
\cite{ParczykPokuttaSpiegelSzabo} have also attracted considerable attention in recent years. Further results on multiplicities of graphs, including stars and complete graphs, can be found in \cite{FranekRodl,Jacobson}.

An ordered graph is a pair $\cG=(G,<_{\cG})$, where $<_{\cG}$ is a linear order on $V(G)$.  The graph $\cK_n$ is the complete graph on $[n]$ with its natural order. An \emph{order-preserving copy} of an ordered graph $\cH$ in $\cG$ is an injective map $\varphi:V(\cH)\to V(\cG)$ that preserves the vertex order and maps every edge of $\cH$ to an edge of $\cG$. The ordered
Ramsey number \(\oR(\cG_1,\ldots,\cG_k)\) is the smallest integer \(N\) such
that every \(k\)-edge-coloring of the ordered complete graph \(\cK_N\)
contains an order-preserving copy of \(\cG_i\) in color \(i\) for some
\(i\in[k]\). Ordered Ramsey numbers have been studied extensively. Choudum and
Ponnusamy~\cite{ChoudumPonnusamy} determined ordered Ramsey numbers for
ordered stars and monotone paths. Balko, Cibulka, Kr\'al, and Kyn\v{c}l
\cite{BalkoCibulkaKralKyncI} and Conlon, Fox, Lee, and Sudakov
\cite{ConlonFoxLeeSudakov} developed general results for ordered Ramsey
numbers. Further results on ordered paths, matchings, and related ordered
structures can be found in
\cite{BalkoJelinekValtr, BalkoPoljak, CoxStolee, MilansStoleeWest, NeidingerWest}.

These developments led us to study the Ramsey multiplicity for ordered graphs. In an off-diagonal problem, the target ordered graphs may have different orders, so their numbers of copies may grow at different rates as the host graph grows. Also, after restricting a coloring to a smaller ordered complete graph, some target graphs may no longer occur. Motivated
by these observations, we introduce a weighted Ramsey multiplicity for ordered graphs and study its asymptotic behavior.

Fix ordered graphs $\cG_1,\ldots,\cG_k$, each containing at least one edge.  For a coloring $\chi:E(\cK_n)\to[k]$, let $N_i(\chi;\cG_i)$ be the number of order-preserving copies of $\cG_i$ whose edges all have the designated color $i$. Set
\(h_i:=v(\cG_i)\), \(e_i:=e(\cG_i)\), \(\hmin:=\min_{i\in[k]}h_i, 
\Imin:=\{i:h_i=\hmin\}\). Let $\blambda=(\lambda_1,\ldots,\lambda_k)\in(0,\infty)^k$.
\begin{definition}\label{def:weighted-multiplicity}
The \emph{weighted ordered multiplicity function} is
\[
\oM_{\blambda}(n;\cG_1,\ldots,\cG_k)
:=\min_{\chi:E(\cK_n)\to[k]}
\sum_{i=1}^{k}\lambda_iN_i(\chi;\cG_i).
\]
If $R=\oR(\cG_1,\ldots,\cG_k)$, the corresponding \emph{weighted threshold multiplicity} is
\[
\oRM_{\blambda}(\cG_1,\ldots,\cG_k)
:=\oM_{\blambda}(R;\cG_1,\ldots,\cG_k).
\]
For $\blambda=\boldsymbol{1}$ we omit the subscript and write $\oM$ and $\oRM$. In the diagonal case, we use the abbreviations $\oR_k(\cG)$ and $\oM_k(n;\cG)$.
\end{definition}

In this paper, we first establish an amplification inequality comparing weighted
ordered multiplicities for host graphs of different orders. It follows
that
\[
\frac{\oM_{\blambda}(n;\cG_1,\ldots,\cG_k)}
     {\binom{n}{\hmin}}
\]
is not a decreasing sequence and converges to a positive finite limit. Hence,
\[
\oM_{\blambda}(n;\cG_1,\ldots,\cG_k)
=\Theta(n^{\hmin}).
\]
We also give a general upper bound obtained from random colorings. For
two-sided ordered stars, we combine the amplification inequality with a
balanced two-interval coloring to obtain lower and upper bounds on their
multiplicity. We then apply the same method to the family
\(\mathfrak B_m\) of ordered perfect matches containing the edge
\(\{1,2m\}\), and obtain the corresponding bounds and a limiting density.
Finally, we use the regularity lemma to find large ordered clusters and show that every correctly colored copy in the reduced coloring gives rise to many order-preserving copies of the corresponding target graph in the original coloring.

\section{Ordered multiplicity and amplification}\label{sec:general}

If $X$ is finite, then $\binom{X}{q}$ is the family of its $q$-element subsets.  We use $v(G)$ and $e(G)$ for the numbers of vertices and edges of a graph $G$, and adopt the convention $\binom{x}{q}=0$ when $q<0$ or $q>x$. An \emph{interval} is a set $I$ such that $x<y<z$ and $x,z\in I$ imply $y\in I$.  Thus $B_1<\cdots<B_t$ denotes consecutive intervals with every vertex of $B_i$ preceding every vertex of $B_j$ for $i<j$.
Since a finite linear order has no nontrivial order-preserving automorphism, it follows that an $h$-element subset of an ordered host supports at most one order-preserving copy of a fixed $h$-vertex ordered graph.  We count copies by their vertex sets.

Let $R=\oR(\cG_1,\ldots,\cG_k)$. For $t\ge R$, define the visible-index set
\begin{equation}\label{eq:Jt-definition}
\Jt:=\{i\in[k]:h_i\le t\}.
\end{equation}
It is nonempty because $R\ge\hmin$.

\begin{theorem}\label{thm:weighted-amplification}
Let $R\le t\le n$.  Every coloring $\chi:E(\cK_n)\to[k]$ satisfies
\begin{equation}\label{eq:weighted-master}
\sum_{i\in\Jt}\lambda_iN_i(\chi;\cG_i)
\binom{n-h_i}{t-h_i}
\ge
\oM_{\blambda}(t;\cG_1,\ldots,\cG_k)\binom{n}{t}.
\end{equation}
Equivalently,
\begin{equation}\label{eq:weighted-normalized-master}
\sum_{i\in\Jt}\lambda_iN_i(\chi;\cG_i)
\frac{\binom{t}{h_i}}{\binom{n}{h_i}}
\ge
\oM_{\blambda}(t;\cG_1,\ldots,\cG_k).
\end{equation}
Consequently,
\begin{equation}\label{eq:weighted-total-amplification}
\oM_{\blambda}(n;\cG_1,\ldots,\cG_k)
\ge
\oM_{\blambda}(t;\cG_1,\ldots,\cG_k)
\frac{\binom{n}{\hmin}}{\binom{t}{\hmin}}.
\end{equation}
Moreover, every coloring has an index $i\in\Jt$ for which
\begin{equation}\label{eq:weighted-color-alternative}
N_i(\chi;\cG_i)
\ge
\frac{\oM_{\blambda}(t;\cG_1,\ldots,\cG_k)}{|\Jt|\lambda_i}
\frac{\binom{n}{h_i}}{\binom{t}{h_i}}.
\end{equation}
\end{theorem}
\begin{proof}
Let $R \le t \le n$, $J_t = \{i\in [k] : h_i \le t\}$, and fix an arbitrary edge-coloring
$\chi:E(K_n)\to [k]$. We count weighted pairs $(X,F)$, where
$X\subseteq [n], |X|=t$ is a $t$-vertex subset, $F$ is an $i$-colored copy of $\mathcal{G}_i$ with $i\in J_t$ fully contained in $X$.
First, fix the subset $X$. Restrict the coloring $\chi$ to $X$ to obtain $\chi|_X:E(K_t)\to [k]$.
By Definition~\ref{def:weighted-multiplicity}, each restriction has weighted target count at least $\oM_{\blambda}(t;\cG_1,\ldots,\cG_k)$. We have
$$
\sum_{i\in J_t}\lambda_i N_i(\chi|_X;\mathcal{G}_i) \ge 
\oM_{\blambda}(t;\cG_1,\ldots,\cG_k).
$$
Indices $i\notin J_t$ satisfy $h_i>t$, so no copy of $\mathcal{G}_i$ can fit inside $X$, i.e., $N_i(\chi|_X;\mathcal{G}_i)=0$. Sum over all $\binom{n}{t}$ distinct $t$-element subsets $X$, we have
\[\sum_{\substack{X\subseteq [n]\\|X|=t}} \sum_{i\in J_t}\lambda_i N_i(\chi|_X;\mathcal{G}_i) \ge \binom{n}{t}\cdot \oM_{\blambda}(t;\cG_1,\ldots,\cG_k).\]
Now count weighted pairs $(X,F)$. Fix an $i$-colored copy $F$ of $\mathcal{G}_i$ (for $i\in J_t$). The copy $F$ occupies exactly $h_i$ vertices. To form a $t$-vertex set $X$ containing $F$, we choose $t-h_i$ additional vertices from the remaining $n-h_i$ vertices, giving $\binom{n-h_i}{t-h_i}$ choices for $X$. There are $N_i(\chi;\mathcal{G}_i)$ such copies of $\mathcal{G}_i$ under coloring $\chi$. The left-hand side rewrites as
$\sum_{i\in J_t}\lambda_i N_i(\chi;\mathcal{G}_i)\binom{n-h_i}{t-h_i}$.
Combining the two expressions yields
\[\sum_{i\in J_t}\lambda_i N_i(\chi;\mathcal{G}_i)\binom{n-h_i}{t-h_i} \ge \oM_{\blambda}(t;\cG_1,\ldots,\cG_k)\binom{n}{t},\]
which proves \eqref{eq:weighted-master}.

For $0\le h\le t$, we have 
\[
\frac{\binom{n}{t}}{\binom{n-h}{t-h}}
=
\frac{\binom{n}{h}}{\binom{t}{h}},
\]
so division of \eqref{eq:weighted-master} by $\binom{n}{t}$ gives \eqref{eq:weighted-normalized-master}. Set
\[
r_i:=\frac{\binom{n}{h_i}}{\binom{t}{h_i}}
\qquad (i\in\Jt).
\]
For $0\le h<t$,
\[
\frac{\binom{n}{h+1}/\binom{t}{h+1}}
     {\binom{n}{h}/\binom{t}{h}}
=
\frac{n-h}{t-h}\ge1.
\]
Hence $r_i\ge r_*:=\binom{n}{\hmin}/\binom{t}{\hmin}$.  By  \eqref{eq:weighted-normalized-master}, we have 
\[
\oM_{\blambda}(t;\cG_1,\ldots,\cG_k)
\le \sum_{i\in\Jt}\frac{\lambda_iN_i(\chi;\mathcal{G}_i)}{r_i}
\le \frac{1}{r_*}\sum_{i=1}^{k}\lambda_iN_i(\chi;\mathcal{G}_i),
\]
is true for arbitrary edge-coloring $\chi$.
Minimizing the final weighted count over $\chi$, we have 
\[
\oM_{\blambda}(t;\cG_1,\ldots,\cG_k)
\le \frac{1}{r_*}\min_{\chi:E(\cK_n)\to[k]}\sum_{i=1}^{k}\lambda_iN_i(\chi;\mathcal{G}_i)= \frac{1}{r_*}\oM_{\blambda}(n;\cG_1,\ldots,\cG_k),
\]
which proves \eqref{eq:weighted-total-amplification}. Finally, one of the $|\Jt|$ summands in \eqref{eq:weighted-normalized-master} is at least their average, which is \eqref{eq:weighted-color-alternative}.
\end{proof}

By Theorem \ref{thm:weighted-amplification}, we obtain the following two corollaries, which are useful in our later proofs.
\begin{corollary}\label{cor:threshold-amplification}
For $R\le t\le n$,
\begin{equation}\label{eq:unweighted-amplification}
\oM(n;\cG_1,\ldots,\cG_k)
\ge
\left\lceil
\oM(t;\cG_1,\ldots,\cG_k)
\frac{\binom{n}{\hmin}}{\binom{t}{\hmin}}
\right\rceil.
\end{equation}
In particular, for every $n\ge R$,
\begin{equation}\label{eq:threshold-total}
\oM(n;\cG_1,\ldots,\cG_k)
\ge
\left\lceil
\oRM(\cG_1,\ldots,\cG_k)
\frac{\binom{n}{\hmin}}{\binom{R}{\hmin}}
\right\rceil.
\end{equation}

\end{corollary}

\begin{proof}
By taking $\blambda=1$ in Theorem~\ref{thm:weighted-amplification} equation \eqref{eq:weighted-total-amplification}, we have equation \eqref{eq:unweighted-amplification}.
Setting $t=R$ in \eqref{eq:unweighted-amplification} gives \eqref{eq:threshold-total}.
\end{proof}

\begin{corollary}\label{cor:weighted-density}
For $n\ge R$, the sequence
\[
\frac{\oM_{\blambda}(n;\cG_1,\ldots,\cG_k)}{\binom{n}{\hmin}}
\]
is not decreasing and bounded above by
\(\lambdamin:=\min_{j\in\Imin}\lambda_j\). Therefore, the limit
\begin{equation}\label{eq:weighted-density-definition}
\pi_{\blambda}(\cG_1,\ldots,\cG_k)
:=\lim_{n\to\infty}
\frac{\oM_{\blambda}(n;\cG_1,\ldots,\cG_k)}{\binom{n}{\hmin}}
\end{equation}
exists and satisfies
\begin{equation}\label{eq:weighted-density-bounds}
0<
\frac{\oRM_{\blambda}(\cG_1,\ldots,\cG_k)}{\binom{R}{\hmin}}
\le
\pi_{\blambda}(\cG_1,\ldots,\cG_k)
\le\lambdamin.
\end{equation}
In particular,
\begin{equation}\label{eq:weighted-theta}
\oM_{\blambda}(n;\cG_1,\ldots,\cG_k)=\Theta(n^{\hmin}).
\end{equation}
\end{corollary}

\begin{proof}
Apply \eqref{eq:weighted-total-amplification} with any $R\le t\le n$ and divide its both sides by $\binom{n}{\hmin}$. This proves monotonicity. Since
$I_* = \{i\in [k] : h_i = h_*\}$ and $\lambda_* = \min_{j\in I_*}\lambda_j$, there exists some index $j\in I_*$ satisfying $\lambda_j = \lambda_*$. Since $j\in I_*$, we have $h_j = h_*$. Define $\chi_0(e) = j$ for every edge $e\in E(K_n)$. This is a valid $k$-edge-coloring of $K_n$. For any  index $i\in [k]$ with $i\neq j$, every edge receives color $j$, so there cannot exist any copy of $\mathcal{G}_i$ whose edges are all color $i$. Therefore, $N_i(\chi_0;\mathcal{G}_i) = 0$ for all $i\neq j$. Only the $i=j$ term contributes to the weighted sum
\[\sum_{i=1}^k \lambda_i N_i(\chi_0;\mathcal{G}_i) = \lambda_j N_j(\chi_0;\mathcal{G}_j).\]
Note that $\mathcal{G}_j$ is an ordered graph on $h_j = h_*$ vertices. For any $h_*$-vertex subset $S\subseteq V(K_n)$, the induced subgraph $K[S]$ is a complete ordered graph on $h_*$ vertices, which contains exactly one order-preserving isomorphic copy of $\mathcal{G}_j$.
The total number of such $h_*$-vertex subsets is $\binom{n}{h_*}$, so $N_j(\chi_0;\mathcal{G}_j) = \binom{n}{h_*}$. Substitute $N_j(\chi_0;\mathcal{G}_j) = \binom{n}{h_*}$ and $\lambda_j = \lambda_*$. We have
$$\sum_{i=1}^k \lambda_i N_i(\chi_0;\mathcal{G}_i) = \lambda_* \binom{n}{h_*}.$$
The minimum over all colorings cannot exceed the value attained by the specific coloring $\chi_0$, hence
\[\oM_{\blambda}(n;\cG_1,\ldots,\cG_k) \le \sum_{i=1}^k \lambda_i N_i(\chi_0;\mathcal{G}_i) = \lambda_* \binom{n}{h_*}.\]
Divide both sides by $\binom{n}{h_*}>0$,
we have $$\frac{\overline{\mathrm{M}}_\lambda(n)}{\binom{n}{h_*}} \le \lambda_*.$$ The lower bound of \eqref{eq:weighted-density-bounds} is clear. Taking $t=R$ in Theorem~\ref{thm:weighted-amplification}, we have 
\begin{equation*}
\frac{\oM_{\blambda}(n;\cG_1,\ldots,\cG_k)}{\binom{n}{\hmin}}
\ge \frac{\oRM_{\blambda}(\cG_1,\ldots,\cG_k)}{\binom{R}{\hmin}}.
\end{equation*}
Since for $n\ge R$, the sequence
\[\frac{\oM_{\blambda}(n;\cG_1,\ldots,\cG_k)}{\binom{n}{\hmin}}\] is not decreasing and is bounded by \(\lambdamin\), we have the upper bound of \eqref{eq:weighted-density-bounds}.
Since
\[\dbinom{n}{h_*} \sim \dfrac{n^{h_*}}{h_*!},\]
it follows from \eqref{eq:weighted-density-bounds} that there exist constants $c_1,c_2>0$ such that
$c_1 n^{h_*} \le \oM_{\blambda}(n;\cG_1,\ldots,\cG_k) \le c_2 n^{h_*}$.
Hence, 
$\oM_{\blambda}(n;\cG_1,\ldots,\cG_k) = \Theta(n^{h_*})$, which implies \eqref{eq:weighted-theta}.
\end{proof}
We also study random-coloring of ordered Ramsey multiplicity, and obtain the following theorem. 
\begin{theorem}\label{thm:weighted-random}
For every $n$,
\begin{equation}\label{eq:weighted-random-upper}
\oM_{\blambda}(n;\cG_1,\ldots,\cG_k)
\le
\min_{\substack{p_i\ge0\\\sum_i p_i= 1}}
\sum_{i=1}^{k}\lambda_i p_i^{e_i}\binom{n}{h_i}.
\end{equation}
Thus,
\begin{equation}\label{eq:weighted-density-random}
\pi_{\blambda}(\cG_1,\ldots,\cG_k)
\le
\min_{\substack{p_i\ge0\ (i\in\Imin)\\\sum_{i\in\Imin}p_i=1}}
\sum_{i\in\Imin}\lambda_i p_i^{e_i}.
\end{equation}
In particular, if \(e_i=e>1\) for every \(i\in I_*\), then
\begin{equation}\label{eq:weighted-common-edge-upper}
\pi_{\blambda}(\cG_1,\ldots,\cG_k)
\le
\left(
\sum_{i\in\Imin}\lambda_i^{-1/(e-1)}
\right)^{1-e},
\end{equation}
and for $e=1$, the right-hand side of \eqref{eq:weighted-density-random} equals $\min_{i\in\Imin}\lambda_i$.
\end{theorem}

\begin{proof}
Fix a probability vector $(p_1,\ldots,p_k)$ and color each edge independently, using color $i$ with probability $p_i$. Every $h_i$-element subset supports one candidate order-preserving copy of $\cG_i$, and its required $e_i$ edges all receive color $i$ with probability $p_i^{e_i}$. Linearity of expectation gives
\[
\mathbb E\!\left[\sum_i\lambda_iN_i\right]
=\sum_i\lambda_i p_i^{e_i}\binom{n}{h_i}.
\]
Since the expectation is the average value of the weighted multiplicity over all colorings, there exists at least one coloring $\chi_0$ satisfying
\[\sum_{i=1}^k \lambda_i N_i(\chi_0;\mathcal{G}_i) \le \mathbb{E}\left[\sum_{i=1}^k \lambda_i N_i\right].\]
By definition, $\oM_{\blambda}(n;\cG_1,\ldots,\cG_k) = \min_{\chi}\sum_{i=1}^k\lambda_i N_i(\chi;\mathcal{G}_i)$, so
$\oM_{\blambda}(n;\cG_1,\ldots,\cG_k) \le \sum_{i=1}^k \lambda_i p_i^{e_i} \binom{n}{h_i}$.
The inequality is valid for all nonnegative $p_i$, thus \eqref{eq:weighted-random-upper} follows.

For \eqref{eq:weighted-density-random}, if $i\in \Imin$, then let $p_i$ be positive probability, otherwise let $p_i=0$. 
Divide by $\binom{n}{\hmin}$ in \eqref{eq:weighted-random-upper}, and let $n\to\infty$, we can obtain \eqref{eq:weighted-density-random}.

Assume now that all relevant edge counts equal $e>1$. By H\"older's inequality, we have 
\begin{align*}
1=\sum_{i\in\Imin}p_i
&=\sum_{i\in\Imin}
(\lambda_i^{1/e}p_i)\lambda_i^{-1/e}\\
&\le
\left(\sum_{i\in\Imin}\lambda_i p_i^e\right)^{1/e}
\left(\sum_{i\in\Imin}\lambda_i^{-1/(e-1)}\right)^{(e-1)/e}.
\end{align*}
Let
\[
p_i=
\frac{\lambda_i^{-1/(e-1)}}
{\sum_{j\in\Imin}\lambda_j^{-1/(e-1)}},
\]
and
$S = \sum_{j\in I_*}\lambda_j^{-1/(e-1)}$. Then
$p_i= \frac{\lambda_i^{-1/(e-1)}}{S}$.
First, verify the normalization constraint:
\[\sum_{i\in I_*}p_i^* = \frac{1}{S}\sum_{i\in I_*}\lambda_i^{-1/(e-1)} = \frac{S}{S}=1.\]
Since
\[\begin{aligned}
\sum_{i\in I_*}\lambda_i p_i^e &= \sum_{i\in I_*}\lambda_i \left(\frac{\lambda_i^{-1/(e-1)}}{S}\right)^e
= \frac{1}{S^e}\sum_{i\in I_*}\lambda_i \cdot \lambda_i^{-e/(e-1)} \\
&= \frac{1}{S^e}\sum_{i\in I_*}\lambda_i^{1-\frac{e}{e-1}}
= \frac{1}{S^e}\sum_{i\in I_*}\lambda_i^{-1/(e-1)} \\
&= \frac{S}{S^e} = S^{1-e},
\end{aligned}\]
by \eqref{eq:weighted-density-random}, we have
\[\pi_{\blambda}(\cG_1,\ldots,\cG_k)
\le
\min_{\substack{p_i\ge0\ (i\in\Imin)\\\sum_{i\in\Imin}p_i=1}}
\sum_{i\in\Imin}\lambda_i p_i^{e_i}\leq \sum_{i\in I_*}\lambda_i p_i^e= S^{1-e}.\]

If $e = 1$, then the objective in \eqref{eq:weighted-density-random} becomes the linear function $\sum_{i\in I_*} \lambda_i p_i$. Let
\(\lambda_{\min} := \min_{j\in I_*} \lambda_j~\text{and}~L := \{i\in I_* : \lambda_i = \lambda_{\min}\}\).
Since $\lambda_i \ge \lambda_{\min}$ for every $i\in I_*$, we have
\[
\sum_{i\in I_*} \lambda_i p_i \ge \lambda_{\min} \sum_{i\in I_*} p_i = \lambda_{\min}.
\]
The equality holds whenever the probability vector $p$ is supported on $L$. Equivalently, the minimum is attained by
taking 
\(p_{i}=1\) for an index \(i\in L\) and \(p_j=0\) for all index \(j\in I_*-i\). Hence, the minimum value is $\min_{i\in I_*} \lambda_i$.
\end{proof}

\section{Ordered Ramsey multiplicity of stars}\label{sec:stars}

For integers $r,s\ge1$, let $\cS_{r,s}$ be the ordered star whose center
has $r-1$ leaves to its left and $s-1$ leaves to its right. Let
\(r_1,s_1,r_2,s_2\ge2\), and
\(\hmin=\min\{r_1+s_1-1,r_2+s_2-1\}, e_i=r_i+s_i-2\), and $R=\oR(\cS_{r_1,s_1},\cS_{r_2,s_2})$. Fix positive weights $\blambda=(\lambda_1,\lambda_2)$.
\begin{theorem}\label{thm:star-multiplicity}
For every $R\le t\le n$,
\begin{equation}\label{eq:weighted-star-lower}
\oM_{\blambda}(n;\cS_{r_1,s_1},\cS_{r_2,s_2})
\ge
\oM_{\blambda}(t;\cS_{r_1,s_1},\cS_{r_2,s_2})
\frac{\binom{n}{\hmin}}{\binom{t}{\hmin}}.
\end{equation}
Moreover,
\begin{equation}\label{eq:weighted-star-random-upper}
\oM_{\blambda}(n;\cS_{r_1,s_1},\cS_{r_2,s_2})
\le
\min_{0\le p\le1}
\left[
\lambda_1p^{e_1}\binom{n}{r_1+s_1-1}
+
\lambda_2(1-p)^{e_2}\binom{n}{r_2+s_2-1}
\right].
\end{equation}
If \(t\ge \max\{r_1+s_1-1,r_2+s_2-1\}\), then every red--blue coloring contains, for some
\(i\in\{1,2\}\), at least
\[
\frac{\oM_{\blambda}(t;\cS_{r_1,s_1},\cS_{r_2,s_2})}
{2\lambda_i}
\frac{\binom{n}{r_i+s_i-1}}{\binom{t}{r_i+s_i-1}}
\]
copies of \(\cS_{r_i,s_i}\) in color \(i\).
\end{theorem}
\begin{proof}
Equation \eqref{eq:weighted-star-lower} is Theorem~\ref{thm:weighted-amplification} with $k=2$, $\cG_1=\cS_{r_1,s_1}$ and $\cG_2=\cS_{r_2,s_2}$. 
Fix any \(p\in[0,1]\), and color each edge of \(\cK_n\) independently red with probability $p$ and blue with probability \(1-p\). Let \(\mathbf{N}_1\) count the red ordered copies of \(\cS_{r_1,s_1}\) and \(\mathbf{N}_2\) count the blue ordered copies of \(\cS_{r_2,s_2}\) in this random coloring. A fixed candidate copy of \(\cS_{r_1,s_1}\) has \(e_1\) edges and is monochromatic red with probability \(p^{e_1}\); there are \(\binom{n}{r_1+s_1-1}\) such vertex subsets, so linearity of expectation gives \(\mathbb{E}[\mathbf{N}_1] = p^{e_1}\binom{n}{r_1+s_1-1}\). Similarly, a fixed candidate \(\cS_{r_2,s_2}\) has \(e_2\) edges and is all blue with probability \((1-p)^{e_2}\), so \(\mathbb{E}[\mathbf{N}_2] = (1-p)^{e_2}\binom{n}{r_2+s_2-1}\). The expected weighted total monochromatic copy count satisfies
\[\mathbb{E}\big[\lambda_1 \mathbf{N}_1 + \lambda_2 \mathbf{N}_2\big] = \lambda_1 p^{e_1}\binom{n}{r_1+s_1-1} + \lambda_2 (1-p)^{e_2}\binom{n}{r_2+s_2-1}.\]
Since the expectation is the average weighted count over all colorings, there exists at least one deterministic red-blue coloring whose weighted monochromatic count is no larger than this expectation. By definition, \(\overline{\mathrm{M}}_\lambda(n;\cS_{r_1,s_1},\cS_{r_2,s_2})\) is the minimal weighted count across all colorings, so
\[\overline{\mathrm{M}}_\lambda(n;\cS_{r_1,s_1},\cS_{r_2,s_2}) \le \lambda_1 p^{e_1}\binom{n}{r_1+s_1-1} + \lambda_2 (1-p)^{e_2}\binom{n}{r_2+s_2-1}.\]
As this bound holds for every \(p\in[0,1]\), we take the minimum over $p$ on the right-hand side to obtain \eqref{eq:weighted-star-random-upper}.

Finally, when $t\ge\max\{r_1+s_1-1,r_2+s_2-1\}$, the visible-index set is $\{1,2\}$. By \eqref{eq:weighted-color-alternative}, we have every red--blue coloring contains either at least
\[\frac{\oM_{\blambda}(t;\cS_{r_1,s_1},\cS_{r_2,s_2})}{2\lambda_1}
\frac{\binom{n}{r_1+s_1-1}}{\binom{t}{r_1+s_1-1}}\] red copies of $\cS_{r_1,s_1}$, or at least
\[\frac{\oM_{\blambda}(t;\cS_{r_1,s_1},\cS_{r_2,s_2})}{2\lambda_2}
\frac{\binom{n}{r_2+s_2-1}}{\binom{t}{r_2+s_2-1}} \] blue copies of $\cS_{r_2,s_2}$.
\end{proof}

For $h\ge2$, define
\begin{equation}\label{eq:balanced-block-count}
B_h(n):=
\binom{\lfloor n/2\rfloor}{h}
+
\binom{\lceil n/2\rceil}{h}.
\end{equation}

\begin{lemma}\label{lem:balanced-split}
For integers $h\ge2$ and $0\le x\le n$,
\[
B_h(n)\leq \binom{x}{h}+\binom{n-x}{h}.
\]
\end{lemma}

\begin{proof}
Define
\[
F(x) = \binom{x}{h} + \binom{n-x}{h}, \qquad 0 \le x \le n.
\]
Clearly, \(F(x) = F(n-x) \). Hence, it suffices to consider the case \( 0 \le x \le \lfloor n/2 \rfloor \).
For any \( 0 \le x \le \lfloor n/2 \rfloor - 1 \), we have \( x \le n - x - 2 \). Since \( \binom{m+1}{k} - \binom{m}{k} = \binom{m}{k-1} \), we compute the difference
\[
\begin{aligned}
F(x+1)-F(x)
&= \binom{x+1}{h} + \binom{n-x-1}{h} - \binom{x}{h} - \binom{n-x}{h} \\[0.2cm]
&= \binom{x}{h-1} - \binom{n-x-1}{h-1}.
\end{aligned}
\]
Since \( x \le n - x - 2 \), it follows that \( x < n - x - 1 \). Since the binomial coefficient \( \binom{m}{h-1} \) is non-decreasing with respect to \(m\), we get
\[
\binom{x}{h-1} \le \binom{n-x-1}{h-1}.
\]
Therefore,
\[
F(x+1) - F(x) \le 0,
\]
which means \( F(x) \) is non-increasing on the interval \( 0 \le x \le \lfloor n/2 \rfloor \). Hence,
\[
F(x) \ge F(\lfloor n/2 \rfloor), \qquad 0 \le x \le \lfloor n/2 \rfloor.
\]

If \( x > \lfloor n/2 \rfloor \), then by symmetry \( F(x) = F(n-x) \), where \( n-x \le \lfloor n/2 \rfloor \). Thus the same inequality holds for all \( 0 \le x \le n \). Hence,
\[
\binom{x}{h} + \binom{n-x}{h} 
\ge 
\binom{\lfloor n/2 \rfloor}{h} + \binom{\lceil n/2 \rceil}{h} 
= B_h(n).
\]
This completes the proof.
\end{proof}

\begin{theorem}\label{thm:balanced-star-upper}
For every pair of two-sided ordered stars,
\begin{equation}\label{eq:weighted-balanced-star-finite}
\oM_{\blambda}(n;\cS_{r_1,s_1},\cS_{r_2,s_2})
\le
\min\{\lambda_1B_{h_1}(n),\lambda_2B_{h_2}(n)\},
\end{equation}
where $h_i=r_i+s_i-1$ for $i\in [2]$.
Moreover, the balanced split minimizes the resulting monochromatic count for the internal color. Therefore,
\begin{equation}\label{eq:weighted-balanced-star-density}
\pi_{\blambda}(\cS_{r_1,s_1},\cS_{r_2,s_2})
\le
\min_{\{i:h_i=\hmin\}}\lambda_i\,2^{1-\hmin}.
\end{equation}
\end{theorem}
\begin{proof}
Fix $j\in\{1,2\}$. Split $[n]$ into consecutive intervals $A<B$ of sizes $x$ and $n-x$. Color all edges inside $A$ and inside $B$ with color $j$, and color every cross-edge with the other color. A two-sided star in the cross color cannot be centered in $A$, since it would need a cross-colored left edge inside $A$; it cannot be centered in $B$, since it would need a cross-colored right edge inside $B$. Hence, the target assigned the cross color has no copy. As every color-$j$ target lies entirely in one block, the weighted count equals
\[
\lambda_j\left[\binom{x}{r_j+s_j-1}+\binom{n-x}{r_j+s_j-1}\right].
\]
Lemma~\ref{lem:balanced-split} shows that this is minimized by the balanced split and equals $\lambda_jB_{r_j+s_j-1}(n)$. Taking the better of $j=1$ and $j=2$ proves \eqref{eq:weighted-balanced-star-finite}.

We now derive the density bound \eqref{eq:weighted-balanced-star-density}. For fixed $h$, evaluate the asymptotic ratio
\[\frac{\binom{\lfloor n/2\rfloor}{h}}{\binom{n}{h}} = \prod_{j=0}^{h-1} \frac{\lfloor n/2\rfloor - j}{n-j}.\]
As $n\to\infty$, each factor converges to $\tfrac12$, so the product converges to $2^{-h}$. Since $B_h(n)=\binom{\lfloor n/2\rfloor}{h}+\binom{n-\lfloor n/2\rfloor}{h}$, we obtain
$$\frac{B_h(n)}{\binom{n}{h}} \to 2^{-h} + 2^{-h} = 2^{1-h}.$$
Take $i\in I_*$ satisfying $r_i+s_i-1=h_*$. Divide inequality \eqref{eq:weighted-balanced-star-finite} by $\binom{n}{h_*}$, we have
\[\pi_\lambda(\mathcal{S}_{r_1,s_1},\mathcal{S}_{r_2,s_2}) = \lim_{n\to\infty}\frac{\overline{\mathrm{M}}_\lambda(n;\mathcal{S}_{r_1,s_1},\mathcal{S}_{r_2,s_2})}{\binom{n}{h_*}} \le \lambda_i \cdot 2^{1-h_*}.\]
Minimizing over all $i\in I_*$ with minimal order $r_i+s_i-1=h_*$ yields \(\pi_\lambda(\mathcal{S}_{r_1,s_1},\mathcal{S}_{r_2,s_2}) \le \min_{\{i:h_i=h_*\}} \lambda_i 2^{1-h_*}\).
\end{proof}

Combining the amplification lower bound with the balanced two-interval
upper bound yields the following estimate for the limiting density.
\begin{corollary}\label{cor:star-density}
In the unweighted case,
\begin{equation}\label{eq:star-density-bounds}
\frac{\oRM(\mathcal{S}_{r_1,s_1},\mathcal{S}_{r_2,s_2})}{\binom{R}{\hmin}}
\le
\pi_{\bf{1}}(\mathcal{S}_{r_1,s_1},\mathcal{S}_{r_2,s_2})
\le
2^{1-\hmin}.
\end{equation}
If $r_1+s_1-1=r_2+s_2-1=h$, then $2^{1-h}$ is half of the uniform random-coloring bound $2^{2-h}$.
\end{corollary}

\begin{proof}
For \eqref{eq:star-density-bounds}, the lower bound comes from Corollary~\ref{cor:weighted-density} equation \eqref{eq:weighted-density-bounds} with unit weights. The upper bound follows from Theorem~\ref{thm:balanced-star-upper}.
When $r_1+s_1-1=r_2+s_2-1=h$, each target has $h-1$ edges, by the random coloring bound Theorem~\ref{thm:weighted-random}, the uniform probability vector $p_1=p_2=\tfrac12$ gives the density upper bound
$\sum_{i=1}^2 p_i^{h-1} = 2\cdot \left(\frac12\right)^{h-1} = 2^{2-h}$.
Comparing the two bounds, we see that $2^{1-h}$ is exactly one half of the uniform random-coloring bound $2^{2-h}$.
\end{proof}



\section{Ordered Ramsey multiplicity of perfect matching}\label{sec:matchings}

\begin{definition}
For $m\ge1$, let $\bB_m$ be the family of all ordered perfect matchings on $[2m]$ containing the edge $\{1,2m\}$. Let $\oR(\bB_m)$ be the least $N$ such that every red--blue coloring of $\cK_N$ contains a monochromatic order-preserving copy of some member of $\bB_m$. Let $\oM_2(n;\bB_m)$ be the minimum total number of monochromatic copies of members of $\bB_m$ in a red--blue coloring of $\cK_n$.
\end{definition}

\begin{lemma}\label{lem:B-family-size}
For $m\ge1$,
\[
|\bB_m|
=
\frac{(2m-2)!}{2^{m-1}(m-1)!}.
\]
\end{lemma}

\begin{proof}
The edge $\{1,2m\}$ is fixed. This problem is equivalent to counting the number of ways to partition the $2(m-1)$ vertices into $m-1$ disjoint vertex pairs. Therefore, the total number of such partitions is computed as
\[
\frac{\dbinom{2m-2}{2}\dbinom{2m-4}{2}\cdots\dbinom{2}{2}}{(m-1)!}
=\frac{1}{(m-1)!}\cdot\frac{(2m-2)!}{(2m-4)!2!}\cdot\frac{(2m-4)!}{(2m-6)!2!}\cdots 1
=\frac{(2m-2)!}{2^{m-1} (m-1)!}.
\]
\end{proof}

\begin{theorem}\label{thm:B-family}
For $m\ge1$, let $R_m=\oR(\bB_m)$ with $R_m\le t\le n$. Then
\begin{equation}\label{eq:B-family-amplification}
\oM(n; \bB_m,\bB_m)
\ge
\oM(t;\bB_m,\bB_m)
\frac{\binom{n}{2m}}{\binom{t}{2m}}.
\end{equation}
Moreover,
\begin{equation}\label{eq:B-family-upper}
\oM(n; \bB_m,\bB_m)
\le
\binom{n}{2m}
\frac{(2m-2)!}{2^{2m-2}(m-1)!}.
\end{equation}
Therefore,
\[
\frac{\oM(n; \bB_m,\bB_m)}{\binom{n}{2m}},
\]
$n\ge R_m$, is not decreasing and converges to a density $\pi(\bB_m)$ satisfying
\begin{equation}\label{eq:B-family-density-bounds}
\frac{\oM(R_m; \bB_m,\bB_m)}{\binom{R_m}{2m}}
\le
\pi(\bB_m)
\le
\frac{(2m-2)!}{2^{2m-2}(m-1)!}.
\end{equation}
In particular, $\oM(n; \bB_m,\bB_m)=\Theta(n^{2m})$ for fixed $m$.
\end{theorem}
\begin{proof}
Equation \eqref{eq:B-family-amplification} can be directly obtained by taking $\blambda=1$ and $\cG_1=\cG_2=\bB_m$ in Equation \eqref{eq:weighted-total-amplification}.

For the upper bound, assign every edge of \(\cK_n\) independently to red or blue, each with probability \(\tfrac12\). On a fixed $2m$-element vertex set there are $|\bB_m|$ candidate matchings, and each is monochromatic with probability $2^{1-m}$. By linearity of expectation and  Lemma~\ref{lem:B-family-size}, the expected number of monochromatic \(\mathcal{B}_m\) copies inside the fixed $2m$-vertex set is
\[
\frac{(2m-2)!}{2^{m-1}(m-1)!}
2^{1-m}
=
\frac{(2m-2)!}{2^{2m-2}(m-1)!}.
\]
There are \(\binom{n}{2m}\) distinct $2m$-vertex subsets of the $n$ vertices of \(\cK_n\). Again apply linearity of expectation across all subsets. The global expected total number of monochromatic \(\mathcal{B}_m\) copies in the random $2$-coloring of \(\cK_n\) is \[\binom{n}{2m}
\frac{(2m-2)!}{2^{2m-2}(m-1)!}.\]
An integer-valued count no larger than the expectation exists, thus \eqref{eq:B-family-upper} holds.

Dividing \eqref{eq:B-family-amplification} by \(\binom{n}{2m}\) yields \[\frac{\oM(n;\bB_m,\bB_m)}{\binom{n}{2m}}
\ge
\frac{\oM(t;\bB_m,\bB_m)}{\binom{t}{2m}},
\] whenever \(R_m \le t \le n\), so the sequence \(a_n = \overline{\mathrm{M}}(n;\mathcal{B}_m,\mathcal{B}_m)/\binom{n}{2m}\) is not decreasing for \(n\ge R_m\). Taking $t=R_m$ in equation \eqref{eq:B-family-amplification}, since $a_n$ is not decreasing, we have \[\frac{\oM(R_m;\bB_m,\bB_m)}{\binom{R_m}{2m}} \le \pi(\mathcal{B}_m).\] By \eqref{eq:B-family-upper}, we have \[\pi(\mathcal{B}_m)\leq \frac{(2m-2)!}{2^{2m-2}(m-1)!}.\]
Since \(\binom{n}{2m}\sim \frac{n^{2m}}{(2m)!}\), by \eqref{eq:B-family-upper} again, 
there exist constants $c_1,c_2>0$ such that
$$c_1 n^{2m} \le \oM(n;\bB_m,\bB_m) \le c_2 n^{2m}.$$
Therefore, we have $\overline{\mathrm{M}}_2(n;\mathfrak{B}_m) = \Theta(n^{2m})$ for fixed $m$.
\end{proof}

\section{Regularity in ordered colorings}\label{sec:regularity}

Let $\chi:E(K_n)\to[c]$ be a $c$-edge-coloring.  For disjoint nonempty sets $A,B\subseteq[n]$ and color $\gamma\in[c]$, write
\[
d_\gamma(A,B):=\frac{e_\gamma(A,B)}{|A||B|}.
\]
The color densities sum to one.

\begin{definition}\label{def:colored-regular-pair}
A pair $(A,B)$ is \emph{$\epsilon$-regular} if, whenever $X\subseteq A$ and $Y\subseteq B$ satisfy $|X|\ge\epsilon|A|$ and $|Y|\ge\epsilon|B|$,
\(|d_\gamma(X,Y)-d_\gamma(A,B)|\le\epsilon~\text{for every }\gamma\in[c]\).
\end{definition}
An \emph{equitable partition with exceptional class} is a partition
\(\mathcal P=(V_0,V_1,\ldots,V_K)\),
in which $|V_1|=\cdots=|V_K|$. We call \(V_1,\ldots,V_K\) the \textit{clusters} of the partition. The partition is $\epsilon$-regular if $|V_0|<\epsilon n$ and all but at most $\epsilon K^2$ unordered pairs of clusters are $\epsilon$-regular.
We first list the following theorem which is useful in our later proofs. 
\begin{theorem}[Szemer{\'e}di \cite{Szemeredi}]\label{thm:multicolor-regularity}
For every $c\ge2$, every integer $m\ge1$, and every $0<\epsilon<1/2$, there are integers
\(M=M(c,m,\epsilon), n_0=n_0(c,m,\epsilon) \)
with the following property.  Let $n\ge n_0$, let $\chi:E(K_n)\to[c]$, and let $\mathcal A$ be an equitable partition of $[n]$ into at most $m$ parts.  Then there is an $\epsilon$-regular equitable partition
\(\mathcal P=(V_0,V_1,\ldots,V_K)\), such that every cluster is contained in a part of $\mathcal A$ and
\(m\le K\le M \).
\end{theorem}

We next prove several useful lemmas for our later proofs.
\begin{lemma}\label{lem:typical-degree}
Let $(A,B)$ be $\epsilon$-regular in a fixed color $\gamma$ with density $d_\gamma(A,B)$.  Fewer than $\epsilon|A|$ vertices $x\in A$ satisfy
\(d_\gamma(x,B)<(d_\gamma(A,B)-\epsilon)|B|\),
and fewer than $\epsilon|A|$ satisfy
\(d_\gamma(x,B)>(d_\gamma(A,B)+\epsilon)|B| \).
The same holds after interchanging $A$ and $B$.
\end{lemma}

\begin{proof}
Let $X\subseteq A$ be the set of vertices that for any $x\in X$, satisfy  $d_\gamma(x,B)<(d_\gamma(A,B)-\epsilon)|B|$. If $|X|\ge\epsilon|A|$, then \(d_\gamma(X,B)<d_\gamma(A,B)-\epsilon \), contradicting regularity. The proof of upper bound is identical, and symmetry gives the assertions for vertices of $B$.
\end{proof}

\begin{lemma}\label{lem:slicing}
Suppose $(A,B)$ is $\epsilon$-regular and $A'\subseteq A$, $B'\subseteq B$ satisfy
\(|A'|\ge\alpha|A| \), \(|B'|\ge\alpha|B|\), and $\alpha\ge\epsilon$.
Then for every color $\gamma$, \(|d_\gamma(A',B')-d_\gamma(A,B)|\le\epsilon \),
and $(A',B')$ is $(2\epsilon/\alpha)$-regular.
\end{lemma}

\begin{proof}
The density estimate follows directly from regularity because $A'$ and $B'$ are large enough. Let
\(\epsilon'=2\epsilon/\alpha \).
If $X\subseteq A'$ and $Y\subseteq B'$ satisfy $|X|\ge\epsilon'|A'|$ and $|Y|\ge\epsilon'|B'|$, then
\(|X|\ge2\epsilon|A|\ge\epsilon|A| \), 
\(|Y|\ge2\epsilon|B|\ge\epsilon|B|\).
For every color $\gamma$, we have 
\begin{align*}
|d_\gamma(X,Y)-d_\gamma(A',B')|
&\le
|d_\gamma(X,Y)-d_\gamma(A,B)|+
|d_\gamma(A',B')-d_\gamma(A,B)|\\
&\le2\epsilon\le\epsilon'.
\end{align*}
Thus $(A',B')$ is $\epsilon'$-regular.
\end{proof}

For $h\ge1$ and $0<d\le1$, define constants recursively by
\begin{align}
\eta_1(d)&=1,
&\delta_1(d)&=1,
\label{eq:counting-base}\\
\eta_h(d)&=
\min\left\{
\frac d4,\frac{1}{4h},
\frac d4\eta_{h-1}(d/2)
\right\},
&
\delta_h(d)&=
\frac12\delta_{h-1}(d/2)
\left(\frac d2\right)^{h-1}
\quad(h\ge2).
\label{eq:counting-recursion}
\end{align}

\begin{lemma}\label{lem:colored-counting}
Let $\cH$ be an ordered graph on $1<\cdots<h$, let $\gamma\in[c]$, and let
\(V_1<\cdots<V_h\) be pairwise disjoint sets.  Suppose that, whenever $ij\in E(\cH)$, the pair $(V_i,V_j)$ is $\eta_h(d)$-regular and has color-$\gamma$ density at least $d$.  Then the number of color-$\gamma$, order-preserving copies of $\cH$ using one vertex from each $V_i$ is at least
\begin{equation}\label{eq:canonical-counting}
\delta_h(d)\prod_{i=1}^{h}|V_i|.
\end{equation}
\end{lemma}

\begin{proof}
We use induction on $h$. When $h=1$, every choice of one vertex from $V_1$ is a canonical copy, so the number of copies is $|V_1|=\delta_1(d)|V_1|$. Assume $h\ge2$.
Let $J\subseteq[h-1]$ be the set of neighbors of $h$ in $\cH$. By Lemma~\ref{lem:typical-degree}, for each $j\in J$, fewer than $\eta_h(d)|V_h|$ vertices $x\in V_h$ have fewer than
\((d-\eta_h(d))|V_j|\ge\frac d2|V_j|\)
color-$\gamma$ neighbors in $V_j$, since $\eta_h(d)\le d/4$. By the union bound and $\eta_h(d)\leq \frac{1}{4h}$, at least
\[
(1-|J|\eta_h(d))|V_h|
\ge
(1-(h-1)/(4h))|V_h|
>
\frac12|V_h|
\]
vertices of $V_h$ are typical for every $j\in J$, that is, $x\in V_h$ is typical if $x$ has at least \((d-\eta_h(d))|V_j|\) color-$\gamma$ neighbors in $V_j$ for every $j\in J$.

Fix such a typical vertex $x$.  For $j<h$, set
\[
V_j'=
\begin{cases}
N_\gamma(x)\cap V_j,&j\in J,\\
V_j,&j\notin J.
\end{cases}
\]
Then $|V_j'|\ge \frac{d}{2}|V_j|$ for $j\in J$, and $V_j'=V_j$ otherwise. Consider an edge $ij$ of the induced ordered graph $\cH-h$. Both $V_i'$ and $V_j'$ have size at least $\frac{d}{2}$ times the corresponding original class. By Lemma~\ref{lem:slicing}, the density of color-$\gamma$ is at least
\[
d_\gamma(V_i',V_j') \ge d - \eta_h(d) \ge d - \frac{d}{4}\geq \frac{d}{2}.
\]
and that the pair is $ \frac{4 \eta_h(d)}{d}$-regular. By the $\eta_h(d)$ recursive definition, we have 
\(\eta_h(d) \le \frac{d}{4}\eta_{h-1}\left(\frac{d}{2}\right) \) and hence \(\frac{4\eta_h(d)}{d} \le \eta_{h-1}\left(\frac{d}{2}\right)\).
Hence, the pair $(V'_i,V'_j)$ is $\eta_{h-1}(\frac{d}{2})$-regular. Consider $\cH-h$ with density parameter $\frac{d}{2}$. By induction hypothesis, there are at least
\[
\delta_{h-1}\left(\frac{d}{2}\right)
\prod_{j<h}|V_j'|
\ge
\delta_{h-1}\left(\frac{d}{2}\right)
\left(\frac d2\right)^{|J|}
\prod_{j<h}|V_j|
\]
canonical copies extending to $x$. Since $\frac{d}{2}\le1$ and $|J|\le h-1$, this is at least
\[
\delta_{h-1}(\frac{d}{2})
\left(\frac d2\right)^{h-1}
\prod_{j<h}|V_j|.
\]

The number of typical vertices in $V_h$ exceeds $\frac12 |V_h|$. Summing the lower bound on canonical copies extending to each such typical vertex $x$, we obtain the total number of monochromatic ordered copies of $\mathcal{H}$, which are greater than 
\[\frac12 |V_h| \cdot \delta_{h-1}\left(\frac{d}{2}\right) \left(\frac{d}{2}\right)^{h-1} \prod_{j<h}|V_j|.\]
Rearranging the product terms separates the product of all cluster sizes, and thus the total copies are greater than 
\[\underbrace{\frac12 \delta_{h-1}\left(\frac{d}{2}\right) \left(\frac{d}{2}\right)^{h-1}}_{\delta_h(d)} \cdot \prod_{j=1}^h |V_j|.\]
The recursive definition
$\delta_h(d) = \frac12 \delta_{h-1}\left(\frac{d}{2}\right) \left(\frac{d}{2}\right)^{h-1}$ yields \(\delta_h(d)\prod_{j=1}^{h}|V_j|\) monochromatic $\mathcal{H}$ copies, as required.
\end{proof}

We next extract an ordered complete set of regular clusters from a regular refinement of consecutive intervals.

\begin{lemma}\label{lem:ordered-regular-clique}
Let $t\ge2$, $q=32t$, and suppose $0<\epsilon\le \frac{1}{16t}$. Let \(I_1<\cdots<I_q \)
be an equitable partition of $[n]$ into consecutive intervals and $n\ge q$. Let
\( \mathcal P=(V_0,V_1,\ldots,V_K)\)
be an $\epsilon$-regular equitable refinement with $K\ge q$, every cluster contained in some $I_j$. Then $\mathcal P$ contains clusters
\(W_1<\cdots<W_t\)
lying in distinct intervals \(I_j~(1\leq j\leq q)\) and forming pairwise $\epsilon$-regular pairs.
\end{lemma}

\begin{proof}
The partition $\mathcal{P}$ has an exceptional set $V_0$ with $|V_0|<\varepsilon n$, and all non-exceptional clusters $V_1,\dots,V_K$ have equal size $s$. We have
\[
    s = \frac{n-|V_0|}{K} \ge \frac{n-\varepsilon n}{K} \ge \frac{n}{2K},
\]
where we use $\varepsilon\le 1/4$. Since the intervals $I_j$ are equally sized partitions of $[n]$ and $n\ge q$, we have $|I_j|\le \frac{2n}{q}$. Let $g_j$ denote the number of regular clusters contained in interval $I_j$. Then
\[
    g_j \le \frac{|I_j|}{s} \le \frac{2n/q}{n/(2K)} = \frac{4K}{q},
\]
so $\max_j g_j \le \frac{ 4K}{q}$.
We now construct a conflict graph $F$ whose vertex set is exactly the set of regular clusters $\{V_1,\dots,V_K\}$. We place an edge between two distinct clusters in $F$ if and only if either the pair is not $\varepsilon$-regular, or the two clusters lie in the same interval $I_j$. By construction, any independent set in $F$ corresponds to a collection of clusters such that every pair is $\varepsilon$-regular and every cluster belongs to a distinct interval $I_j$. Therefore, it suffices to prove that $F$ has an independent set of size at least $t$.
We bound the total number of edges in $F$. The number of non-$\varepsilon$-regular pairs is at most $\varepsilon K^2$ by the property of equitable regular partitions. The number of edges that joining two clusters in the same interval is $\sum_{j=1}^q \binom{g_j}{2}$. Furthermore,
\[
    \sum_{j=1}^q\binom{g_j}{2} \le \frac12\left(\max_j g_j\right)\sum_{j=1}^q g_j \le \frac12\cdot \frac{4K}{q}\cdot K = \frac{2K^2}{q}.
\]
Substituting $\varepsilon\le \frac{1}{16t}$ and $q=32t$, we have 
\[
    e(F) \le \varepsilon K^2 + \frac{2K^2}{q}
    \le \frac{K^2}{16t}+\frac{2K^2}{32t}
    = \frac{K^2}{8t}.
\]
We now prove a general lower bound on the independence number $\alpha(F)$ of any graph on $K$ vertices:
\[
    \alpha(F)\ge \frac{K^2}{2e(F)+K}.
\]
We establish this bound via a random permutation argument. Denote by \(\Omega\) the set of all \(K!\) permutations of \(V(F)\), and let every permutation is equally likely. A random permutation \(\sigma \in \Omega\) is chosen uniformly. For each vertex \(v \in V(F)\), define the event
\(A_v = \{\sigma \in \Omega : \sigma(v) < \sigma(u) \text{ for every } u \in N(v)\} \),
where \(\sigma(v)\) denotes the position of \(v\) in the permutation. Define the indicator random variable \(X_v\) depending on the random permutation \(\sigma\) by
\[
X_v(\sigma) =
\begin{cases} 
1, & \sigma \in A_v, \\
0, & \sigma \notin A_v.
\end{cases}
\]
That is, \(X_v\) is a \(0\)-\(1\) random variable that depends on the random permutation \(\sigma\). For a fixed permutation \(\sigma\), define the random subset
\( S(\sigma) = \{ v \in V(F) : X_v(\sigma) = 1 \}\).
Thus \(S(\sigma)\) consists of those vertices that appear before all their neighbors in the permutation \(\sigma\).  Note that, for every \(\sigma\), the set \(S(\sigma)\) is an independent set in \(F\). Indeed, if two adjacent vertices \(u,v\) were both in \(S(\sigma)\), then one of them, say \(u\), appears before \(v\) in \(\sigma\). Since \(u\) is a neighbor of \(v\), the vertex \(v\) would not appear before all its neighbors, contradicting \(v \in S(\sigma)\). Hence no edge has both endpoints selected. 
Therefore, for every \(\sigma\), \(\alpha(F) \ge |S(\sigma)| \). Taking expectations over the uniform random permutation \(\sigma\),
\begin{equation}\label{equation1}
\alpha(F) \ge \mathbb{E}[|S|]  
\end{equation}
Now, for any fixed \(\sigma\), we have the trivial identity
\(|S(\sigma)| = \sum_{v \in V(F)} X_v(\sigma) \),
because the right-hand side exactly counts the vertices that satisfy \(X_v=1\).  
By linearity of expectation,
\begin{equation}\label{equation2}
\mathbb{E}[|S|] = \sum_{v \in V(F)} \mathbb{E}[X_v] = \sum_{v \in V(F)} \mathbb{P}(A_v).    
\end{equation}
For a fixed vertex \(v\), consider the closed neighbor \(N[v] = N(v) \cup \{v\}\), which has size \(d(v)+1\).   
The event \(A_v\) occurs exactly when \(v\) is the first element among \(N[v]\). Hence
\[
\mathbb{P}(A_v) = \frac{1}{d(v)+1}.
\]
Combining \eqref{equation1} and \eqref{equation2}, we have
\[
\alpha(F) \ge \sum_{v \in V(F)} \frac{1}{d(v)+1}.
\]
We now apply the Cauchy--Schwarz inequality:
\[
    \left(\sum_{v}(d(v)+1)\right)\left(\sum_{v}\frac{1}{d(v)+1}\right)\ge \left(\sum_{v}1\right)^2=K^2.
\]
Rearranging yields
\[
    \sum_{v}\frac{1}{d(v)+1}\ge \frac{K^2}{\sum_{v}(d(v)+1)}.
\]
By the handshaking lemma, $\sum_v d(v)=2e(F)$, so
\(\sum_{v}(d(v)+1)=2e(F)+K \).
We obtain the independence bound
\[
    \alpha(F)\ge \frac{K^2}{2e(F)+K}.
\]
Substitute our edge bound $e(F)\le K^2/(8t)$ and use $K\ge q=32t$:
\[
    2e(F)+K \le \frac{K^2}{4t}+\frac{K^2}{32t}=\frac{9K^2}{32t}.
\]
Therefore,
\[
    \alpha(F)\ge \frac{K^2}{\frac{9K^2}{32t}}=\frac{32t}{9}>t.
\]
Hence $F$ contains an independent set of size at least $t$. By the definition of the conflict graph, these $t$ clusters are pairwise $\varepsilon$-regular and lie in pairwise distinct intervals $I_j$. Ordering them by their interval positions yields the desired sequence $W_1<W_2<\dots<W_t$. This completes the proof.
\end{proof}

A reduced coloring assigns to each selected cluster pair a color of maximum density.  Since the densities sum to one, every selected color has density at least $1/c$.

\begin{theorem}\label{thm:regular-lifting}
Let \(c\ge2\), let \(\mathcal G_1,\ldots,\mathcal G_c\) be fixed
ordered graphs, and let \(\blambda\in(0,\infty)^c\).
Fix \(t\ge R\), and set
\(J_t=\{i\in[c]:h_i\le t\} \), \(\delta_i=\delta_{h_i}(1/c)~(i\in J_t)\).
Then there exist \(\varepsilon,\alpha>0\) and \(n_0\) such that, for
\(n\ge n_0\), every \(c\)-edge-coloring of \(\cK_n\) contains clusters
\(W_1<\cdots<W_t, |W_j|\ge\alpha n
\),
which are pairwise \(\varepsilon\)-regular. Color each reduced edge \(W_uW_v\) by a color of maximum density.
Let \(\mathcal Q_i\) be the family of correctly colored reduced copies
of \(\mathcal G_i\), and let \(T_i^{\rm reg}\) count their color-\(i\) copies of  \(\mathcal G_i\) meeting every selected cluster in at most one vertex. Then
\begin{equation}\label{eq:regularity-reduced-weighted}
T_i^{\rm reg}\ge
\delta_i\sum_{F\in\mathcal Q_i}
\prod_{W_j\in V(F)}|W_j|
\qquad (i\in J_t),
\end{equation}
and
\begin{equation}\label{eq:regularity-color}
\sum_{i\in J_t}
\frac{\lambda_iT_i^{\rm reg}}
{\delta_i(\alpha n)^{h_i}}
\ge
\oM_{\blambda}
(t;\mathcal G_1,\ldots,\mathcal G_c).
\end{equation}
\end{theorem}
\begin{proof}
Let \(\eta_h(d)\) be the regularity constant in the Lemma \ref{lem:colored-counting}. Choose
\[0<\varepsilon\le
\min\left\{
\frac{1}{16t},
\min_{i\in J_t}\eta_{h_i}(1/c)
\right\}, \]  $q=32t$.
Let \(M=M(c,q,\varepsilon)\) and \(N=N(c,q,\varepsilon)\) be the constants given by Theorem \ref{thm:multicolor-regularity}, and set \(\alpha=\frac{1}{2M}\), $n_0=\max\{N,q\}$.

Fix \(n\ge n_0\) and a \(c\)-edge-coloring of \(\cK_n\). Partition
\([n]\) equitably into consecutive intervals
$I_1<\cdots<I_q$.
Applying Theorem~\ref{thm:multicolor-regularity} with this partition and lower bound \(q\)
gives an equitable partition
\([n]=V_0\cup V_1\cup\cdots\cup V_K\),
where \(q\le K\le M \) and $|V_0|<\varepsilon n$, every \(V_\ell\), \(\ell\in[K]\), is contained in some \(I_a\), and all clusters but not $V_0$ have the same size 
\(|V_1|=\cdots=|V_K|=s \). By Lemma~\ref{lem:ordered-regular-clique} there are distinct indices \(a_1,\ldots,a_t\in[K]\) such that the whole clusters
\(V_{a_1},\ldots,V_{a_t}\) lie in distinct intervals and form pairwise
\(\varepsilon\)-regular pairs. Relabeling them according to their
positions, write
$W_1<\cdots<W_t$.
Since each \(W_j\) is one of the original nonexceptional clusters,
\[
|W_j|=s=\frac{n-|V_0|}{K}
\ge \frac{(1-\varepsilon)n}{M}
\ge \frac{n}{2M}
=\alpha n.
\]

Color the edge \(uv\) of the reduced ordered complete graph on \([t]\)
by a color of maximum density between \(W_u\) and \(W_v\). Since the \(c\) color densities sum to one, the selected
density is at least \(1/c\). The resulting reduced graph is a
\(c\)-edge-coloring of \(K_t\). Its number of correctly colored copies
of \(\mathcal G_i\) is \(|\mathcal Q_i|\); moreover,
\(\mathcal Q_i=\varnothing\) whenever \(i\notin J_t\). 
Now fix \(i\in J_t\) and \(F\in\mathcal Q_i\), it is clearly $h_i\leq t$. Let
\(W_{a_1}<\cdots<W_{a_{h_i}}\) be the clusters corresponding to the ordered vertices of \(F\). For
every edge \(uv\in E(\mathcal G_i)\), the corresponding cluster pair is
\(\varepsilon\)-regular and has color-\(i\) density at least \(1/c\).
Since
$\varepsilon\le\eta_{h_i}(1/c)$,
this pair is also \(\eta_{h_i}(1/c)\)-regular. Lemma~\ref{lem:colored-counting} therefore gives at least
\[
\delta_{h_i}(1/c)
\prod_{W_j\in V(F)}|W_j|
=
\delta_i\prod_{W_j\in V(F)}|W_j|
\]
color-\(i\), order-preserving copies of \(\mathcal G_i\).

Distinct reduced copies have distinct ordered cluster, so their
color-\(i\), ordered preserving copies of \(\mathcal G_i\) are disjoint. Summing over \(F\in\mathcal Q_i\) gives
\[
T_i^{\rm reg}
\ge
\delta_i
\sum_{F\in\mathcal Q_i}
\prod_{W_j\in V(F)}|W_j|,
\]
which proves \eqref{eq:regularity-reduced-weighted}.

Finally, since every \(F\in\mathcal Q_i\) uses \(h_i\) selected clusters
and each selected cluster has size at least \(\alpha n\),
\(T_i^{\rm reg}
\ge
\delta_i|\mathcal Q_i|(\alpha n)^{h_i}\).
Furthermore,
\[
|\mathcal Q_i|
\le
\frac{T_i^{\rm reg}}
     {\delta_i(\alpha n)^{h_i}}.
\]
By the definition of the weighted multiplicity, we have 
\(\sum_{i\in J_t}\lambda_i|\mathcal Q_i|
\ge
\oM_{\blambda}
(t;\mathcal G_1,\ldots,\mathcal G_c)\). Multiplying by \(\lambda_i\), summing over \(i\in J_t\), we obtain
\[
\sum_{i\in J_t}
\frac{\lambda_iT_i^{\rm reg}}
     {\delta_i(\alpha n)^{h_i}}
\ge
\oM_{\blambda}
(t;\mathcal G_1,\ldots,\mathcal G_c),
\]
which proves \eqref{eq:regularity-color}.
\end{proof}

\section*{Declaration on the use of AI}
ChatGPT was used for grammar checking, language polishing, and improving the clarity of the exposition. The core ideas and overall proof strategy were developed by the authors. The authors take full responsibility for the accuracy and originality of the paper.

\section*{Data availability}
No datasets were generated or analyzed in this study.

\end{document}